\documentclass[12pt]{article}
\usepackage{a4wide}
\usepackage{latexsym, amsmath, amsfonts, amssymb, amsthm}

\newtheorem{theo}{Theorem}
\newtheorem{prop}[theo]{Proposition}
\newtheorem{lemma}[theo]{Lemma}

\newtheorem{rem}[theo]{Remark}

\newtheorem{fact}[theo]{Fact}

\newcommand{\R}{\ensuremath{\mathbb{R}}} %reels
\newcommand{\Var}{{\rm Var}} %variance

\def\id{\mathop{\rm Id}\nolimits}%Trace
\def\tr{\mathop{\rm tr}\nolimits} %trace
\def\eps{\varepsilon}

\newcommand{\upchi}{\raise1pt\hbox{$\chi$}}

\newcommand{\be}{\begin{equation}}
\newcommand{\ee}{\end{equation}}
\def\benu{\begin{enumerate}}
\def\eenu{\end{enumerate}}

\def\cost{\mathop{\rm \bf c}\nolimits}
\def\ncost{\mathop{\rm \bf \tilde c}\nolimits}
\def\costV{\mathop{\rm \bf c}_V\nolimits}
\def\costmin{{\rm \mathcal N}}

\def\F{\mathcal F}
\def\ds{\, d\sigma}
\def\sint{\int_{S^{n-1}}}
\def\smallint{\mbox{$\int$}}

\title{Transport inequalities for log-concave measures, quantitative forms and applications
}
\author{Dario Cordero-Erausquin 
 }
 \date{June, 2016}      

\begin{document}

\maketitle

\begin{abstract}
We review some simple techniques based on monotone mass transport that allow to obtain transport-type inequalities for any log-concave probability measure, and for more general measures as well. We discuss quantitative forms of these inequalities, with application to the  Brascamp-Lieb variance inequality.
\end{abstract}

\section{Introduction}

Throughout the paper we work, when needed, with some fixed scalar product $\cdot$ and Euclidean norm $|\cdot|$ on  $\R^n$.
Although our main motivation is to analyse log-concave densities, meaning densities of the form $e^{-V}$ with $V$ convex, our result apply to more general situations, regardless of the convexity of the potential $V$. We can often work with a locally Lipschitz function  $V:\R^n \to \R$ with the mild assumption that
\begin{equation}\label{hyp}
\int (1+|x|^2+ |\nabla V(x)|^2) e^{-V(x)}\, dx <+\infty.
\end{equation}
Actually, when $V$ is convex, we don't need these assumptions, but not much is lost by imposing it. Given such $V$, we introduced the probability measure $\mu_V$ defined by 
$$d\mu_V(x):= \frac{e^{-V(x)}}{\int e^{-V}}\, dx.$$
Note that the density is by assumption everywhere strictly positive.

Following Kantorovich's idea, given a function $\cost:\R^n \times \R^n \to \R$ (one  interprets $\cost(x,y)$ as the cost of moving a unit mass from $x$ to $y$ or of bringing back  a unit mass from $y$ to $x$), we can define a transportation cost $\mathcal W_{\cost}$ between two Borel probability measures $\mu$ and $\nu$ on $\R^n$ by
$$\mathcal W_{\cost}(\mu, \nu):= \mathcal W_{\cost(x,y)}(\mu, \nu):= \inf_\pi \iint_{\R^n\times \R^n}  \cost(x,y)\, d\pi(x,y)$$
where the infimum is taken over all probability measures $\pi$ on $\R^n\times \R^n$ projecting on $\mu$ and $\nu$, respectively. 
From the definition of $\mathcal W_{\cost}(\mu, \nu)$ and Fubini's theorem,  we see that it suffises that the cost is well defined on $(\R^n \setminus X )\times \R^n$ only, where $\mu(X)=0$. Under very mild hypothesis on $\cost$, one can prove that there exists a coupling $\pi$ which is optimal, that is which achieves the infimum above (see~\cite[Chapters~4 and~5]{V}).  The cost $\cost(x,y)=|y-x|^p$, $p\in [1, +\infty)$,
is used for the definition of the $L^p$-Kantorovich-Rubinstein (or Wassertein) distance
$$W_p(\mu, \nu):= \big(\mathcal W_{|x-y|^p} (\mu, \nu) \big)^{1/p}.$$

Recall that given  two probability measures $\mu$ and $\nu$ on $\R^n$, the relative entropy of $\nu$ with respect to $\mu$ is defined by
$$H(\nu ||\mu):=
\begin{cases}
\int f\log(f) \, d\mu & \textrm{ if } d\nu(x) = f(x)d\mu(x) \textrm{ with } f\log_+ (f) \in L^1(\mu)\\
+\infty & \textrm{otherwise} 
\end{cases}.$$
Accordingly, we should only consider probability measures that have a density, in short "absolutely continuous" probability measures. 
Recall also that the variance of a function $g\in L^2(\mu)$ is defined by 
$$\Var_\mu(g):= \int \Big(g - \int g\, d\mu\Big)^2 \, d\mu.$$

The inequality in the next Proposition appeared in~\cite{BL} where it was derived in the dual form~\eqref{ic} as a consequence of the Pr\'ekopa-Leindler inequality. By now it is folklore in optimal mass transportation theory and known to most specialists. The investigation of equality cases seems to be new. 

\begin{prop}\label{thm1}
Let $V:\R^n \to \R$ be locally Lipchitz function satisfying~\eqref{hyp}.    Define for every $y$ and almost every $x$ in $\R^n$  the (asymmetric) cost
\begin{equation}\label{defcost2}
\costV(x,y):= V(y) - V(x) - \nabla V(x) \cdot (y-x).
\end{equation}
Then, for every (absolutely continuous) probability measure $\nu$ on $\R^n$ we have 
\begin{equation}\label{ineq:thm1}
\mathcal W_{\costV}(\mu_V, \nu) \le H(\nu ||\mu_V) .
\end{equation}
Moreover, when $V$ is convex,  equality holds if and only if $\nu$ is a translate of $\mu_V$.
\end{prop}

\begin{rem}\label{rem:equality}
Regarding the treatment of equality case, we can prove a sharper result. Namely, we will establish the following statement: Let $V:\R^n \to \R$ be locally Lipchitz function satisfying~\eqref{hyp}, and assume that $\mu_V$ has a positive Cheeger constant $h(\mu_V)>0$ (see definition below; this is the case when $V$ is convex). Then, there is equality in the transport inequality~\eqref{ineq:thm1} of Proposition~\ref{thm1} \emph{if and only if} $V$ is convex and  $\nu$ is a translate of $\mu_V$.

The fact that the convexity of $V$ is necessary for equality cases is reminiscent of the equality cases in the Brunn-Minkowski inequality.
\end{rem}

When $V$ is convex,  it is possible to define the cost for every $x$ (and not just for almost every $x$) by using the  subgradient $\partial V(x)$ at $x$ of the convex function $V$ (see~\cite{rockafellar} for background on subgradients):
$$\partial V(x) := \{w\in \R^n\; ; \ V(x+ h) \ge V(x) + w \cdot h,\  \forall h \in \R^n\}.$$ 
Indeed, the Proposition can then be stated with the following cost $\costV$ in place of~\eqref{defcost2}: 
\begin{equation}\label{defcost1}
\forall (x,y)\in \R^n \times \R^n, \qquad \costV(x,y):= \sup_{w \in \partial V(x)} \Big\{V(y) - V(x) - w \cdot (y-x)\Big\}.
\end{equation}
Recall that $V$ is locally Lipschitz and so differentiable $\mu$-almost-everywhere.

Note that when $V$ is convex,  we have $\costV(x,y) \ge0$ with $\costV(x,x)=0$  in~\eqref{defcost2} and~\eqref{defcost1};  when $V$ is strictly convex, $\costV(x,y)>0$ if $x\neq y$.

Let us mention that by a simple and standard dualization procedure for transportation inequalities (see~\cite{ledoux}), the statement of Proposition~\ref{thm1}  is equivalent to the following infimal convolution inequality: for every (bounded) function $g:\R^n \to \R$,
\begin{equation} \label{ic}
\int e^{Q_{\costV} (g)} \, d\mu_V \le e^{\int g\, d\mu_V}
\end{equation}
where 
\begin{equation} \label{defic}
Q_{\costV} (g) (y) := \inf_{x} \big\{ g(x) + \costV(x,y)\big\}.
\end{equation}
We should also mention that transportation cost inequalities of the form stated above  imply concentration of measure inequalities (for $\costV$-neighborhoods);  we refer to~\cite{ledoux} for details. 

The interest of the  statement in Proposition~\ref{thm1} resides in the fact that no uniform convexity of $V$ is needed. This is reminiscent of the  Brascamp-Lieb variance inequality~\cite{BL:1976} (anticipated in different context by H\"ormander), which states that for a $C^2$ smooth convex function $V:\R^n \to \R$ with $\int e^{-V}<+\infty$ we have, for every locally Lipschitz function $g\in L^2(\mu_V)$,
\begin{equation}\label{hbl}
\Var_{\mu_V} (g) \le \int (D^2V(x))^{-1} \nabla g(x) \cdot \nabla g(x) \, d\mu_V(x).
\end{equation}
Since the cost $\costV(x,y)$ in Proposition~\ref{thm1}  behaves, when $x$ and $y$ are close to each other, like $\frac12 D^2 V(x) (y-x) \cdot (y-x)$, it follows by a standard linearization argument that Proposition~\ref{thm1}   implies the Brascamp-Lieb inequality~\eqref{hbl}. We shall recall the argument later. 

Another interesting feature of Proposition~\ref{thm1}  is that it is an \emph{affinely invariant} statement, in the sense that it does not depend on the Euclidean structure we put on $\R^n$. More precisely, we don't need a scalar product in the statement: the gradient $w=\nabla f(s)$ (or a subgradient) comes from a linear form $\ell=df(x)\in (\R^n)^\ast$, and we can use $\ell(y-x)$ in place of $w\cdot (y-x)$. This reflects also in the fact that the Brascamp-Lieb inequality~\eqref{hbl} shares the same affine invariance: if $\varphi:\R^n \to \R^n$ is an (invertible) affine map, then the functions $V_\varphi = V\circ \varphi^{-1}$ and $g_\varphi= g\circ\varphi^{-1}$ satisfy $\Var_{\mu_{V_\varphi}} (g_\varphi)=\Var_{\mu_V} (g)$ and 
$$ \int (D^2V_\varphi(x))^{-1} \nabla g_\varphi(x) \cdot \nabla g_\varphi(x) \, d\mu_{V_\varphi}(x) = \int (D^2V(x))^{-1} \nabla g(x) \cdot \nabla g(x) \, d\mu_V(x).$$

Other consequences of Proposition~\ref{thm1}  are Talagrand's transportation inequalities for Gaussian like measures. Observe that for the standard Gaussian measure $\gamma$, when $V(x)=|x|^2/2$, we have
$$\costV(x,y)= |y-x|^2/2$$
and the inequality becomes exactly Talagrand's inequality~\cite{tala}: for every probability density $\nu$ on $\R^n$,
\begin{equation}\label{ineqtala}
\frac12 W_2^2(\gamma , \nu)\le  H(\nu || \gamma),
\end{equation}
with equality if and only if $\nu$ is a translate of $\gamma$.  More generally,  if $V$ is $C^2$ with $D^2 V \ge \lambda \id$ on $\R^n$ for some $\lambda>0$, then by second-order Taylor expansion we see that the cost satisfies 
$$\costV(x,y) \ge \lambda |y-x|^2/2,$$ 
and therefore we deduce that in this case, for every probability measure $\nu$ on $\R^n$ we have 
\begin{equation}\label{Ttransport}
\frac{\lambda}2 \, W_{2}^2(\mu_V, \nu) \le H(\nu ||\mu_V) .
\end{equation}
This inequality appeared  in~\cite{BL, OV, blower}. We refer to~\cite{ledoux, GL} for background and references on transportation inequalities.

The proof of Proposition~\ref{thm1}  is very short; it is a minor adaptation of the transportation proof of Talagrand's inequality~\eqref{Ttransport} given in~\cite{cordero}. With a little more effort one can actually prove a quantitative form of the inequality involving a remainder term. To state the result, we need some notation. Given a probability measure $\mu$ on $\R^n$, we denote by $h(\mu)$ the best (i.e. largest)  nonnegative constant for which the inequality
\begin{equation}\label{defcheeger}
h(\mu) \int \Big| g(x) - \int\! g\, d\mu \Big| \, d\mu(x) \le \int |\nabla g|\, d\mu
\end{equation}
holds for every smooth enough $g\in L^1(\mu)$. This constant, up to a factor $2$, is also known as the Cheeger isoperimetric constant. 

When $\mu$ is log-concave,  then it is known that $h(\mu)>0$, and  $h(\mu)^2$ is actually equivalent, up to an universal constant,  to the spectral gap of the Laplacian associated to $\mu$ (or equivalently the inverse of the Poincar\'e constant). More explicitly, if we denote by $\lambda(\mu)$ the best 
nonnegative constant for which the inequality
$$\lambda(\mu) \int \Big| g(x) - \int\! g\, d\mu \Big|^2 \, d\mu(x) \le \int |\nabla g|^2\, d\mu$$
holds for every smooth enough $g\in L^2(\mu)$, then when $\mu$ is a log-concave measure on $\R^n$,  we have
\begin{equation}\label{equivcheegerpoincare}
c\,  h(\mu)^2 \le \lambda(\mu) \le C\, h(\mu)^2
\end{equation} 
for some universal (numerical) constants $c, C>0$, independent of $\mu$ and $n$; see~\cite{ledoux04, Emilman}. 

In the rest of the paper, we will adopt the lazy but convenient tradition from asymptotic functional analysis  to call "a numerical constant $c$" any positive constant larger than $2$ or smaller than $1/2$ ($c$ may even vary from line to line). So a numerical constant refers to a universal constant (in particular it does not depend on $n$, $V$, $\mu$, $\nu$, etc.) whose exact value is irrelevant but who could \emph{a priori} be computed explicitly. 

There is another natural cost function associated to any measure having a positive Cheeger constant, namely the cost $\min\big( h(\mu)^2 |y-x|^2 \, ,\,  h(\mu) |y-x|  \big)= \costmin (h(\mu)|y-x|)$, where
\begin{equation}\label{defcostmin}
\forall t\ge 0, \qquad \costmin(t) := \min(t^2, \, t).
\end{equation}
This cost -in this form or in some equivalent form- has been studied by several authors (see again~\cite{ledoux, GL} for details).  

Since equality holds in Proposition~\ref{thm1}  when $\nu$ is a translate of $\mu$ it is natural, if we want a remainder term, to minimize over translations, or equivalently, to impose some centering. 

The main result of this note is the following Theorem. 

\begin{theo}[General transport inequality with a remainder term]\label{thm2}
Let $V:\R^n \to \R$ be a locally Lipschitz function satisfying~\eqref{hyp} and let $\costV$ be the cost defined by~\eqref{defcost2}.  Introduce the cost
$$\ncost_V (x,y) = \cost_V (x,y) + c\,\costmin\big(h(\mu_V) |y-x| \big)$$
where $c>0$ is an appropriate numerical constant.

Then, for every probability measure $\nu$ on $\R^n$ such that $\int x \, d\nu = \int x \, d\mu_V$ we have 
\begin{equation}\label{ineq:thm2}
\mathcal W_{\ncost_V}(\mu_V, \nu)  \le H(\nu ||\mu_V) .
\end{equation}
As a consequence, we have the quantitative version of Proposition~\ref{thm1} when $\int x \, d\nu = \int x \, d\mu_V$:
\begin{equation}\label{ineq2:thm2}
H(\nu ||\mu_V)  \ge  \mathcal W_{\cost_V}(\mu_V, \nu)  +  c\,  \mathcal W_{\costmin(h(\mu_V) |y-x|)}(\mu_V, \nu)
\end{equation}
and in particular, we also have, 
\begin{equation}\label{ineq3:thm2}
H(\nu ||\mu_V)  \ge \mathcal W_{\cost_V}(\mu_V, \nu)  +  c \min\Big\{ h(\mu_V)^2 W_1^2, (\mu_V, \nu), h(\mu_V) W_1(\mu_V, \nu)\Big\}
\end{equation}
\end{theo}

Note that unlike the quantities $H$ and $\mathcal W_{\costV}$, the cost $\costmin_{\costmin(h(\mu_V) |y-x|)}$ is very much dependent on the scalar product, which should therefore be chosen with care. 

Let us explain how the consequences of~\eqref{ineq:thm2} stated in the Theorem are obtained.  The first one~\eqref{ineq2:thm2} follows from a general and straightforward principle: given two costs $c_1, c_2$ we always have $\mathcal W_{c_1+ c_2}(\cdot, \cdot)  \ge \mathcal W_{c_1}(\cdot, \cdot)  + \mathcal W_{c_2}(\cdot, \cdot) $. The "in particular", may seem more dubious. The reason that~\eqref{ineq3:thm2} follows indeed from~\eqref{ineq2:thm2} is that, up to numerical constants (see below) we can replace the function $\costmin(s)=\min(s^2,s)$ by a convex increasing function $\mathcal F(s)$, and then we can invoque Jensen's inequality to ensure that $\mathcal W_{\mathcal F(|y-x|)}(\nu, \mu) \ge \mathcal F\big( \mathcal W_{|y-x|} (\mu, \nu) \big)$.  

Note however that the form~\eqref{ineq3:thm2} is strictly weaker than the forms~\eqref{ineq:thm2} and~\eqref{ineq2:thm2}. In particular, we should note that  the cost $\costmin(h(\mu_V) |y-x|)$ behaves like $h(\mu_V)^2 |y-x|^2$ when $x$ and $y$ are close to each other, and this behavior is well adapted to linearization procedures.

Let us describe some consequences of Theorem~\ref{thm2} in the case where $V$ is convex.  Applied to Gaussian type measures, when $\costV(x,y) \ge \lambda|x-y|^2/2$, it amounts to a quantitative version of the transport inequality~\eqref{Ttransport}. 

\begin{prop}[Gaussian type transport with a remainder]\label{prop:qT}
Let $V:\R^n \to \R$ be a $C^2$ convex function with $D^2 V \ge \lambda\id $ on $\R^n$ for some $\lambda >0$ (we have mainly in mind the Gaussian measure, for which $\lambda=1$). 
Then, for every probability measure $\nu$ on $\R^n$ such that $\int x \, d\nu = \int x \, d\mu_V$ we have 
\begin{eqnarray}
H(\nu ||\mu_V)  - \frac{\lambda}2 W_{2}^2(\mu_V, \nu) &\ge&  c\,   \mathcal W_{\costmin(h(\mu_V) |y-x|)}(\mu_V, \nu) \label{quantgaussiantransport1}\\
& \ge &\tilde c \min\Big\{ h(\mu_V)^2 W_1^2, (\mu_V, \nu), h(\mu_V) W_1(\mu_V, \nu)\Big\}  \label{quantgaussiantransport2}
\end{eqnarray}
for some numerical constants $c,\tilde c >0$. One can also replace $h(\mu_V)^2$ by $\lambda$ since  $h(\mu_V)^2 \ge c'\, \lambda$ for some numerical constant $c'>0$. 
\end{prop}

Next, linearization of the inequality~\eqref{ineq:thm2} in Theorem~\ref{thm2} leads  to a reinforced Brascamp-Lieb inequality in the case of centered functions.

\begin{prop}\label{propRBL}
Let $V:\R^n \to \R$ be a $C^2$ convex function with $\int e^{-V}<+\infty$.  For every locally Lipschitz function $g\in L^2(\mu_V)$ with
$$\int x \big(g(x) - \int g\, d\mu_V\big)  \,  d\mu_V(x) = 0$$
we have
$$\Var_{\mu_V} (g) \le  \int \big[ D^2V+ c\, h(\mu_V)^2 \id\big]^{-1}\nabla g \cdot \nabla g \, d\mu_V,$$
for some numerical constant $c>0$. We can replace $h(\mu_V)^2$ by $\lambda(\mu_V)$, in view of~\eqref{equivcheegerpoincare}.
\end{prop}

One can derive a similar result using H\"ormander's $L^2$-method,  but with a different centering condition of the form $\int \nabla g \, d\mu_V=0$ (see~\cite{BC}). 

We should add, as apparent from the proof, that the convexity of $V$ is not really needed in Propostion~\ref{propRBL}. The correct assumption is that $D^2V+ c\, h(\mu_V)^2(\mu_V)\id$ is nonnegative.  In particular, the result applies to perturbed log-concave measures, provided $h(\mu_V)>0$. 

Equality cases in the Brascamp-Lieb inequality~\eqref{hbl}  are given, exactly, by the functions  $g$ of the form
\begin{equation}\label{eqBL}
g(x) = \nabla V(x) \cdot v_0 + c_0
\end{equation}
with  $v_0\in \R^n$ and $c_0 \in \R$.  In order to have a nice quantitative version, one would like to get rid of the centering assumption and to measure, in some form, a "distance" 
$$\inf_{v_0,c_0} d(g, \nabla V \cdot v_0 +c_0) $$ 
to the set of extremizers~\eqref{eqBL}. Here is an attempt.

\begin{prop}[Brascamp-Lieb inequality with a remainder term]\label{propQBL}
Let $V:\R^n \to \R$ be a $C^2$ convex function with $\int e^{-V}<+\infty$.
Then for every    locally Lipschitz function  $g\in L^2(\mu_V)$, if we denote
$$
g_0(x):= g(x) - \nabla V(x) \cdot v_0 - c_0, \quad c_0:= \int\! g\, d\mu_V\textrm{ and }  v_0 :=  \int y \, (g(y)-c_0)\, d\mu_V(y),$$
we have 
\begin{multline*}
\int (D^2V(x))^{-1} \nabla g(x) \cdot \nabla g(x) \, d\mu_V(x) - \Var_{\mu_V} (g) \\
\ge c\lambda(\mu_V)\int (D^2V)^{-1} \big( D^2V+ c \lambda(\mu_V)\id)^{-1}\nabla g_0 \cdot \nabla g_0 \, d\mu_V
\end{multline*}
where $c>0$ is a numerical constant. As a consequence, if we denote  by $\lambda_{\rm max}(x)$ the largest eigenvalue of the nonnegative operator $D^2V(x)$, we have
$$\int (D^2V(x))^{-1} \nabla g(x) \cdot \nabla g(x) \, d\mu_V(x) - \Var_{\mu_V} (g) \ge  \frac{c \lambda(\mu_V)}{\sup _x   \lambda_{\rm max}(x) + c \lambda(\mu_V)}  \int |g_0|^2 \, d\mu_V $$
and 
$$\int (D^2V(x))^{-1} \nabla g(x) \cdot \nabla g(x) \, d\mu_V(x) - \Var_{\mu_V} (g) \ge  \frac{\tilde c\,  \lambda(\mu_V)^2}{\int\lambda_{\rm max} (  \lambda_{\rm max} + c \lambda(\mu_V)) \, d\mu_V} \Big(\int |g_0|\, d\mu_V\Big)^2$$
where $\tilde c>0$ is a numerical constant. 
\end{prop}

Let us recall that $\lambda(\mu_V)$ can be estimated by   
$\lambda(\mu_V) \ge \frac{c}{\int \lambda_{\rm min}^{-1}\, d\mu_V}$,
where  $\lambda_{\rm min}(x)$ denotes the lowest eigenvalue of the nonnegative operator $D^2V(x)$ and $c>0$ is a numerical constant (see~\cite{veysseire,  MK}). 
Therefore, the constant in the previous Proposition (which is an increasing function of $\lambda(\mu_V))$ can be lower bounded by some integrals of $\lambda_{\rm min}$ and $\lambda_{\rm max}$ with respect to $\mu_V$. For instance, using the previous bound and the fact $(\int \lambda_{\rm min}^{-1}\, d\mu_V)^{-1} \le \int \lambda_{\rm max}\, d\mu_V$ we find 
 $$ \frac{\tilde c\,  \lambda(\mu_V)^2}{\int\lambda_{\rm max} (  \lambda_{\rm max} + c \lambda(\mu_V)) \, d\mu_V}  \ge \frac{c}{\big(\int \lambda_{\rm min}^{-1}\, d\mu_V\big)^2\, \int \lambda_{\rm max}^2\, d\mu_V }$$
for some numerical constant $c>0$. This might provide computable a constant beyond the easy case where $\lambda \le D^2 V \le R$ on $\R^n$. 

\medskip

We conclude this introduction with some bibliographical comments. Part of the present note is rather elementary, and many arguments are known to specialists in mass transport, some having appeared implicitly or explicitly in recent or older works. 
For instance, we already said that Proposition~\ref{thm1} was folklore in the theory, and while writing these notes we heard about the work of Bolley, Gentil and Guillin~\cite{BGG} which contains an analogue, in a less straightforward form,  of the inequality of Proposition~\ref{thm1}  together with its connection to the Brascamp-Lieb inequality. If we go back in time, the idea of using the remainder term in the transportation proof of~\cite{cordero}  appears, in the case of dimension one, in the paper by Barthe and Kolesnikov~\cite{BK}. Similar arguments in higher dimensions for unconditional measures were recently used in~\cite{klartag14} and in a form very close to the one used here in~\cite{CG}. Mass transport arguments combined with Poincar\'e inequalities (of different nature than the one we use) were put forward to exhibit  remainder terms in isoperimetric type inequalities in the far-reaching work of Figalli, Maggi and Pratelli, in particular in~\cite{FMP} for the case of log-concave measures (or rather convex sets). Our treatment is in part very close to the recent work of Fathi, Indrei and Ledoux~\cite{FIL} were the mass transport remainder term is combined with a Poincar\'e inequality in order to get a bound on the deficit for Talagrand's inequality~\eqref{Ttransport} in the case of the Gaussian measure (they have also similar, but deeper, arguments for the log-Sobolev inequality, a case that was also considered in~\cite{BGRS}). 

The quantitative transport inequality obtained by Fathi, Indrei and Ledoux~\cite{FIL} for the standard Gaussian measure $\mu=\gamma$ on $\R^n$  (a case where $\lambda(\mu)=1$)  is as follows: for any probability measure $\nu$ with $\int x \, d\nu(x) = \int x\, d\gamma(x) = 0$, 
$$H(\nu ||\gamma)  - \frac{1}2 W_{2}^2(\gamma, \nu) \ge c \min\Big( \frac{W_{1,1}^2(\gamma, \nu)}n , \frac{W_{1,1}(\gamma, \nu)}{\sqrt n} \Big)$$
where $W_{1,1} :=\mathcal W_{\|x-y\|_1} $ with $\|x-y\|_1:=\sum_{i=1}^n |x_i-y_i|$. 
If we compare with inequality~\eqref{quantgaussiantransport2} in Proposition~\ref{prop:qT} above applied to the Gaussian measure $\mu=\gamma$, we see that our result is formally stronger, since
$$ W_1(\mu, \nu) \ge \frac{W_{1,1}(\mu, \nu)}{\sqrt n}.$$
Actually, our bound is significantly better in many cases,  but  both bounds are "equally  bad" when $\nu$ is a product of centered measures being at a "large" distance from the one-dimensional Gaussian, since in this case one expects a remainder of order $n$ and both results give something of order $\sqrt n$ (on the other hand, it is not clear to us that this situation is the most relevant one). 

Regarding quantitive versions of the variance Brascamp-Lieb inequality,  Harg\'e~\cite{harge} (by an $L^2$-method) and recently~Bolley, Gentil and Guillin~\cite{BGG} (by linearization of a transport inequality) obtained a remainder term which is,  up to  constants depending on $\mu_V$,  of the form 
$$\Big( \int g\, V\,  d\mu_V  - \int g \, d\mu_V \, \int V \, d\mu_V \Big)^2=:R_V(g)$$
Note that, unlike the remainder term in Proposition~\ref{propQBL}, this term $R_V(g)$ does not vanish only for extremizers. For instance $R_V(g)$  is zero if $V$ is even and $g$ odd. Actually, the space where $R_V$ vanishes is of co-dimension one in $L^2(\mu_V)$ whereas extremizers~\eqref{eqBL} form a $(n+1)$ dimensional subspace. Of course, it could be  that such type of remainder is nonetheless sometimes better and more useful than the one we obtained. 
Bolley, Gentil and Guillin also derive,  in the same work~\cite{BGG} but by a different method (namely by linearization of a functional Brunn-Minkowski inequality), a second quantitive form of the variance Brascamp-Lieb inequality with a remainder term that vanishes exactly for the extremizers~\eqref{eqBL}, as expected. This remainder however is not an $L^1$ or $L^2$ distance to the space of extremizers, and so the comparison with the result of our  Proposition~\ref{propQBL} is not  clear to us. 

\medskip

The plan of the paper is as follows. In the next section we prove Proposition~\ref{thm1}  and Theorem~\ref{thm2}. For this we recall some tools from the Brenier-McCann monotone mass transport theory, and prove a general lower bound for the remainder term (Lemma~\ref{reste}) that might be of independent interest. Then, we prove Proposition~\ref{propRBL} and Proposition~\ref{propQBL}.

We would like to thank Bernard Maurey for useful observations on  our manuscript, and the anonymous referee for several insightful questions and observations that led to improved statements. We would like also to express our deep gratitude to Emanuel Milman for his most sharp reading of the first version of our preprint. He pointed out to us a fatal mistake in the use we made of his results; this led us to rewrite the main statements and their proofs.

\section{Mass transport, minoration of the remainder and proofs of Proposition~\ref{thm1} and Theorem~\ref{thm2}}

The proof of theorems~\ref{thm1} and~\ref{thm2} use monotone transportation of measure in the spirit of~\cite{cordero}. 

Given two probability measures $\mu$ and $\nu$ on $\R^n$ with densities $F$ and $G$, respectively, we know from Brenier~\cite{brenier} and McCann~\cite{mccann95} that there exists a convex function $\psi$ such that the map $\nabla\psi$ pushes forward $\mu$ onto $\nu$.
By the simple but useful weak-regularity theory of McCann \cite{mccann} we have,
for $\mu$-almost any $x$,
\begin{equation}
F(x)=  G(\nabla(\psi(x))  \det D^2 \psi(x). \label{eq_ma} 
\end{equation}
Here $D^2 \psi (x) $ stands for the Hessian
of the convex function $\psi$ in the sense of Aleksandrov, that exists almost everywhere. There are several ways to use this equation to prove our inequalities. One can use the McCann weak theory of change of variables~\cite{mccann}, as in~\cite{cordero}. The advantage is that it relies on simple arguments in convexity and  Lebesgue measure theory. Alternatively, one can use results on the regularity of Monge-Amp\`ere equation, in the spirit of those obtained by Caffarelli. This relies on more difficult and deeper arguments. However, partial regularity results for solutions of Monge-Amp\`ere have been simplified and extended recently, and we shall favor this point of view.  

Let us assume that $\mu$ and $\nu$ are supported on the whole $\R^n$, and that the densities are continuous and strictly positive (so locally bounded above and  away from zero), 
Since the support of the target measure is convex (here $\R^n$), one can prove that the convex function $\psi$  solves the Monge-Amp\`ere equation~\eqref{eq_ma} also in the sense of Aleksandrov (see the argument given in the proof of~\cite[Theorem~3.3]{FdF}, and by the assumption above on the densities, the local regularity of~\cite{mooney}, say, applies. In particular, $\psi$ is $W^{2,1}_{\text{loc}}(\R^n)$.

To prove the transport inequality of Proposition~\ref{thm1}  for $d\mu_V=\frac{e^{-V(x)}}{\int e^{-V}}dx$, we assume that $d\nu= f(x)\, d\mu_V(x)$. It is sufficient to prove the inequalities in Proposition~\ref{thm1} and Theorem~\ref{thm2} in the case where $f$ is continuous and strictly positive on $\R^n$, so that the previous assumptions are satisfied. We can also assume that $\nu$ has second moment. 

We introduce the Brenier map $T=\nabla \psi$ between $\mu_V$ and $\nu$.  We have that  $\psi\in W^{2,1}_{\text{loc}}(\R^n)$ and that almost everywhere
$$e^{-V(x)} = f(T(x))e^{-V(T(x))} \det D^2\psi(x).$$
It is convenient to introduce the displacement $\nabla\theta(x) = T(x)-x= \nabla\psi(x) -x$ (i.e. $\theta(x) := \psi(x) -|x|^2/2$).  If we take the log  in the previous equation and introduce $\costV(x,T(x))= V(T(x)) - V(x) - \nabla V(x) \cdot \nabla \theta(x)$, we find
\begin{multline*}
\log(f(T(x)) )-c(x,T(x))= \nabla V(x) \cdot \nabla \theta(x) - \log\det D^2\psi(x) \\
= \nabla V \cdot \nabla\theta - \Delta \theta + \Delta \theta - \log\det (\id + D^2\theta).
\end{multline*}
We integrate with respect to $\mu_V$. Noticing that $\int \log(f\circ T) d\mu_V=\int f \log(f) d\mu_V$,  we have
$$H(\nu||\mu_V) -  \int \costV(x,T(x))\, d\mu_V=  \int\big[\nabla V \cdot \nabla\theta - \Delta\theta \big]\, d\mu_V + \int \big[ \Delta \theta - \log\det(\id + D^2\theta)\big]\, d\mu_V.$$
The first term in the right-hand side vanishes after integration by parts, thanks to the integrability assumptions we have made. Let us justify this. 
\begin{fact}
$\displaystyle \int \Delta \theta\,  e^{-V} = \int \nabla \theta \cdot \nabla V \, e^{-V}$.
\end{fact}
\begin{proof}
Since $\int |\nabla \psi|^\alpha \, e^{-V} = \int |y|^\alpha \, d\nu(y)$ and $\nu$ as second moment, we have in view of our assumptions that  $\int |\nabla \theta|^2\, e^{-V} <+\infty$ 
 and $\int |\nabla \theta|\, e^{-V} <+\infty$. Let $h$ be  $C^1$ function on $\R^n$, with values on $[0,1]$,  that is compactly supported and is identically one in a neighborhood of $0\in \R^n$. Introduce the sequence $h_k(x) = h(x/k)$. We have $0 \le h_k \le 1$, $h_k(x) \uparrow 1$ for every $x\in \R^n$ and $\|\nabla h_k\|_\infty \to 0$ as $k\to +\infty$. We have
 $$\int h_k \Delta \theta  \, e^{-V} = -\int \nabla h_k \cdot \nabla \theta \, e^{-V} + \int h_k \nabla \theta \cdot \nabla V  \, e^{-V}.$$
 For the left-hand side, we wan write 
  $$\int h_k \Delta \theta  \, e^{-V}  = \int h_k \Delta \psi  \, e^{-V} - n \int  h_k e^{-V}$$
  and each term converges using the monotone convergence theorem (since $\Delta \psi \ge 0$), giving $ \int \Delta \psi  \, e^{-V} - n \int  e^{-V}=  \int( \Delta \psi - n)   \, e^{-V} = \int \Delta \theta  \, e^{-V}$.
  The first term in the right-hand side tends to zero, since it is bounded by $\|\nabla h_k\|_\infty \int |\nabla \theta|  \, e^{-V}$. For the last term, we conclude by using the dominated convergence theorem, since
  $$2 \int |\nabla \theta \cdot \nabla V|   \, e^{-V} \le \int |\nabla \theta |^2  \, e^{-V}  + \int |\nabla V|^2  \, e^{-V} < +\infty.$$
 \end{proof}

So we have arrived to following elementary formula:
\begin{eqnarray}
H(\nu||\mu_V) & = & \int \cost_V(x,T(x))\, d\mu_V(x) +  \int \big[ \Delta \theta - \log\det(\id + D^2\theta)\big]\, d\mu_V \nonumber \\
&=& \int \cost_V(x,T(x)) \, d\mu_V(x) +\int \big[ \tr D^2\theta - \tr(\log(\id + D^2\theta))\big]\, d\mu_V \nonumber \\
&=& \int \cost_V(x, T(x) )\, d\mu_V(x) + \int \tr(\F(D^2\theta)) \, d\mu_V,\label{mainbound}
\end{eqnarray}
where $\F:[-1, +\infty[\to [0, +\infty[$ stands for the convex (increasing on $\R^+$)   function defined by
\begin{equation}\label{defF1}
\mathcal F(t):=t-\log(1+t), \qquad t\in \R^+.
\end{equation}

Since by definition $ \int \cost_V(x, T(x) )\, d\mu_V(x)  \ge \mathcal W_{\costV}(\mu_V, \nu)$ and $\F \ge 0$, we have  proved in particular the inequality in Proposition~\ref{thm1}. 

The treatment of the cases of equality in Proposition~\ref{thm1} requires a bit of extra work (in particular since $T$ was not \emph{a priori} the $\costV$-optimal map); we postpone it to the end of the present section and go on with the proof of Theorem~\ref{thm2}.

In order to prove Theorem~\ref{thm2}, we have to play a bit with second term in the right-hand side of~\eqref{mainbound}, as it is done in the works we mentioned in the introduction. Indeed, a ``remainder" term of this form 
appears in several mass transport proofs (for instance~\cite{FMP, BK, klartag14, CG, FIL}),  sometimes in equivalent forms such as $\sum (|s_i|+ \frac1{1+|s_i|}-1)$ or $\sum \frac{s_i^2}{1+|s_i|}$ (here $s_i$ refer to the eigenvalues of $D^2\theta$). 
Anyway, the crucial property of the these functions and of the  convex function $t-\log(1+t)$ is that it behaves like $t^2$ for $t$ close to zero, and like $t$ for $t$ large. More precisely, we have
for every $t\in ]-1,+\infty[$ that $F(t) \ge F(|t|)$ and that every $s\ge 0$ 
\begin{equation}\label{prop0}
\frac14 \min(s^2, s) \le \mathcal \F(s) \le \min(s^2, s).
\end{equation}
But we find it more convenient to work with the convex function $\F(|t|)$ rather  than with $\costmin(|t|)=\min(t^2, |t|)$. 

The treatment of the remainder term is stated in the next, central,  Lemma, which is of independent interest. 

\begin{lemma}\label{reste}
Let $\mu$ be a probability mesure on $\R^n$ absolutely continuous with respect to the Lebesgue measure  and $\theta\in W^{2,1}_{\rm{loc}}(\R^n)$ with $D^2\theta \ge -\id$ almost everywhere. We assume that
$|\nabla\theta|\in L^1(\mu)$ with $\int \nabla \theta \, d\mu = 0$. Then,
\begin{equation}\label{goal}
\int \tr(\F(D^2\theta)) \, d\mu \ge c \int \F\big(h(\mu) |\nabla \theta|\big)\, d\mu
\end{equation}
for some numerical constant $c>0$.
\end{lemma}

Note that our assumption $\int x \, d\mu_V = \int x \, d\nu$ rewrites as  $\int \nabla \theta\, d\mu_V = 0$ so if we use in~\eqref{mainbound} the previous Lemma with $\mu=\mu_V$ and $\theta$ our displacement function, we find
$$H(\nu||\mu_V)  \ge  \int \ncost_V(x, T(x)) \, d\mu_V \ge W_{\ncost_V}(\mu_V, \nu)$$
as claimed in Theorem~\ref{thm2}. 

So it only remains to prove the previous Lemma. 

Denote by $\sigma$ the uniform probability measure on $S^{n-1}$. Recall that for every vector $X\in \R^n$ we have
$$n \sint (X\cdot u)^2 \ds(u) = |X|^2$$
and that 
\begin{equation}\label{smean}
c|X|\le \sqrt n \sint |X\cdot u |\ds(u) \le |X|,
\end{equation}
for some numerical constant $c>0$.

We will use the following Fact, the proof of which is postponed below. 

\begin{fact} \label{trace}
Let $A$ be a symmetric matrix with eigenvalues $>-1$. Then
$$
\tr(\F(A)) \ge \frac18 \sint \F\big(\sqrt n |Au| \big)  \, \ds(u) .
$$
\end{fact}

We will combine this with the following isoperimetric type inequality. It is due to Bobkov and Houdr\'e~\cite{BH}, where it is stated with the median. We will include a proof below for completeness. 

\begin{fact}[Bobkov-Houdr\'e]\label{BH}
Let $\mu$ be a probability measure on $\R^n$. For every regular enough $f:\R^n\to \R$ we have
\begin{equation}\label{ineq:BH}
\int \F\big( \big|f - \smallint f\, d\mu\big|\big) \, d\mu \le c \int \F\Big(\frac1{h(\mu)} |\nabla f|\Big) \, d\mu
\end{equation}
for some numerical constant $c>0$ (for instance $c=3\times 4^3$ works). 
\end{fact}

With these two facts in hand, we can now finish the proof of~\eqref{goal}. We have, by Fact~\ref{trace} and Fubini's theorem, that
\begin{equation}\label{step}
\int \tr(\F(D^2\theta)) \, d\mu\ge  \frac18\sint\!\!\int \F\big(\sqrt n |D^2\theta u|\big)\, d\mu \ds (u) 
\end{equation}

For any fixed vector $u \in S^{n-1}$ we have that the  function $g(x)=h(\mu)\sqrt n\, \nabla \theta(x) \cdot u$ is  $W^{1,1}_{\rm{loc}}(\R^n)$ with derivative  $\nabla g(x) =h(\mu)\sqrt n \,  (D^2 \theta(x)) u$, and $\int g\, d\mu_V=0$. So we have by applying Fact~\ref{BH} that
$$\int \F\big(\sqrt n |(D^2\theta) u|\big)\, d\mu
\ge \frac1{3\times 4^3} \int \F(h(\mu) \sqrt n \,  |\nabla\theta \cdot u | )\, d\mu.$$
Back to~\eqref{step}, integrating the previous inequality with respect to $\ds(u)$, using that $\F$ is convex and~\eqref{smean} we find
$$\int \tr(\F(D^2\theta)) \, d\mu \ge \frac1{3\times 4^3\times 8} \int \F\Big(h(\mu) \sqrt n \sint |\nabla\theta\cdot u | \ds(u)\Big)d\mu \ge  c \int \F(h(\mu) |\nabla \theta|)\, d\mu.$$

This ends the proof of Lemma~\ref{reste}, modulo the two facts above that we prove now. 

\begin{proof}[Proof of Fact~\ref{trace}]
Let us collect first some straightforward properties of  $\F$, or equivalently, in view of~\eqref{prop0},  of $\min(s^2, |s|$).  These functions commute with power functions. In particular, we shall use that
\begin{equation}\label{prop1}
\forall s\ge 0 , \qquad \F(\sqrt s ) \le \sqrt{\F (s)} \le 2 \F(\sqrt s ). 
\end{equation}
Note  that   $\F(2 s) \le 4 \F(s)$. Observe also that for a finite family $s_1, \ldots , s_k \ge 0$ we have
\begin{equation}\label{prop2}
\sum_{i\le k} \F(s_i) \ge \frac14\F\Big( \sqrt{\sum_{i\le k} s_i^2}\Big).
\end{equation}
Indeed, if we denote by $s_{\max}$ the largest number, we see using~\eqref{prop0} that
$$\sum_{i\le k} \F(s_i) \ge \frac14  \sum_{i\le k} \min(s_{i}, s_{i}^2) \ge 
\frac14 \min(s_{\max}, s_{\max}^2) \sum_{i\le k} \Big(\frac{s_i}{s_{\max}}\Big)^2 ;$$
Then, distinguishing between  $s_{\max} \le 1$, and  $s_{\max} \ge 1$, a case for which we replace it by using  $ s_{\max} \le \sqrt{\sum_{i\le k} s_i^2}$, we find
$$\sum_{i\le k} \F(s_i)  \ge  \frac14 \min\Big(\sum_{i\le k} s_i^2, \,  \sqrt{\sum_{i\le k} s_i^2}\Big) \ge \frac14\F\big( \sqrt{\sum_{i\le k} s_i^2}\big) .$$
Let us mention that inequality~\eqref{prop2} will be the one responsible for the loss of a factor $\sqrt n$ in the case of product measures. 

Back to the proof of the Fact, let us notice  that $\F(A) \ge \F(|A|)$, since $\F(s_i) \ge \F(|s_i|)$ for any eigenvalue $s_i$ of $A$. Denote 
$$H:=|A| = \sqrt{A^\ast A};$$ it is a nonnegative symmetric matrix. 
Let  $(e_1, \ldots, e_n)$ be an orthonormal basis of eigenvectors of $A$. Then 
$$\tr(\F(H)) = \sum_{i\le n} |\F(H) e_i|  = \sum_{i\le n} \F(|H e_i|) \ge \frac12\sum_{i\le n} \sqrt \F( |H e_i|^2) $$
where we used~\eqref{prop1}. Let us mention in passing that using the convexity of $\F$, we can establish more generally that for for any $v\in S^{n-1}$ we have
$|\F(H)v | \ge \frac12 \sqrt{\F(|Hv|^2)}.$

From this, the fact that $\sqrt \F$ is \underline{concave} on $\R^+$, then again~\eqref{prop1} and finally~\eqref{prop2} we find
\begin{multline*}
\tr(\F(H)) \ge  \frac12 \sum_{i\le n}  \sqrt\F\big(|He_i|^2\big) = \frac12 \sum_{i\le n}  \sqrt\F\Big(n \sint (He_i \cdot u)^2\ds(u)\Big) \\
 \ge   \frac12\sint \sum_{i\le n}  \sqrt\F(n (He_i\cdot u)^2) \, ds(u) 
\ge  \frac12 \sint \sum_{i\le n}  \F(\sqrt n |He_i\cdot u|) \, ds(u)  \\
\ge  \frac18  \sint \F\Big(\sqrt n \sqrt{\sum_{i\le n}|He_i\cdot u|^2}\Big) \, ds(u)  = \frac18 \sint \F\big(\sqrt n |Hu| \big)  \, ds(u) .
\end{multline*}
To conclude, use that $|Hu|^2 = H^2 u \cdot u = A^2 u \cdot u = |Au|^2$. 

\end{proof}

\begin{proof}[Proof of Fact~\ref{BH}]
By scaling the metric, we can assume that $h(\mu)=1$. More precisely, we can change the scalar produce $x\cdot y$ into $h(\mu)^{-1} x \cdot y$, which changes the gradient accordingly in~\eqref{defcheeger} and~\eqref{ineq:BH}. 

Denote by $m$ a $\mu$-median of $f$. By a standard argument,  it is enough to prove that 
$$4 \int \F \big(|f-m|\big)\, d\mu \le 3\times 4^2\int \F(|\nabla f |)\, d\mu.$$
Indeed, since $\F(2t ) \le 4\F(t)$ and since $\F$ is convex increasing on $\R^+$, we have for any function $g$ with $\mu$-median $m_g$:
$$\int \F\big(\big| g - \smallint g\, d\mu\big|\big) \, d\mu \le 2 \int \F (|g-m_g|) \, d\mu + 2  \F\big(\big|\smallint (g-m_g)\, d\mu \big|\big) \le 4 \int \F(|g-m_g|)\, d\mu .$$
The same kind of argument shows that one can use a median $m_g$ instead of the mean, in the definition~\eqref{defcheeger}. Indeed, for any $g\in L^1(\mu)$ with median $m_g$ we have
$$\int |g-m_g| \, d\mu \le \int \big| g - \smallint g\, d\mu\big| \, d\mu  +  \big| \smallint g\, d\mu  - m_g \big| .$$
We can assume that $m_g \ge \smallint g \, d\mu$ (otherwise use $-g$), and by the definition of $m_g$ and by Markov's inequality we have
$$\frac12 \le \mu(\{ g \ge m_g\}) \le \mu(\{ |g - \smallint g\, d\mu  | \ge m_g - \smallint g\, d\mu \})\le \frac1{m_g - \smallint g \, d\mu } \int \big| g- \smallint g\, d\mu \big|\, d\mu ,$$
and so $\big|m_g - \smallint g\, d\mu \big| \le 2 \int \big|g - \smallint g \, d\mu \big| \, d\mu$. Therefore, we have
$$\int |g-m_g| \, d\mu \le 3 \int \big|g - \smallint g \, d\mu \big| \, d\mu  \le 3 \int |\nabla g |\, d\mu.$$

Given our $f$ with $\mu$-median $m$, let us introduce the (continuous) function $g$ such that
$$g(x)=\begin{cases} 
\F(|f(x)-m|) & \text{ if } f(x) \ge m \\
-\F\big(|f(x)-m|\big) & \text{ if } f(x) < m.
\end{cases}
$$
Since $\F \ge 0$, the function $g$ has zero $\mu$-median. Therefore 
\begin{equation}\label{cheeger}
\int \F(|f-m|)\, d\mu = \int |g| \, d\mu \le 3 \int |\nabla g | \, d\mu .
\end{equation}
We will now use an argument inspired by~\cite{klartag14}. Let us observe that for every $s\in \R^+$, $t\in [0,1]$ (this is only good choice to estimate the Legendre transform of $\F$), we have
$$st \le  4\F(s) + \frac1{16} t^2 .$$
Indeed, for $s\ge 1$ the inequality is obvious since $4\F(s) \ge s \ge st$, and for $s<1$ use that $4\F(s) \ge s^2$ to complete the square.  Since $\F' \in [0,1]$ on $\R^+$, we have
\begin{multline*}
\int |\nabla g|\, d\mu =  \int \F'(|f-m|)\,  |\nabla f |\, d\mu \le  4 \int \F(|\nabla f|) \, d\mu + \frac1{16} \int \F'(|f-m|)^2\, d\mu \\
\le 4  \int \F(|\nabla f|) \, d\mu + \frac1{4} \int \F(|f-m|)\, d\mu,
\end{multline*}
where the second inequality follows from $\F'(s)^2 \le 4 \F(s)$ for every $s\in \R^+$ (this can can be seen, for instance, by computing $(4\F-\F'^2)'(s)=\frac{2s(1+2s + 2s^2)}{(1+s)^3} \ge 0$). 
Plugging this in~\eqref{cheeger} we find
$$\frac14\int \F(|f-m|)\, d\mu  \le  3\times 4  \int \F(|\nabla f|) \, d\mu,$$
which gives the desired inequality.
\end{proof}

This completes the proof of Lemma~\ref{reste} and Theorem~\ref{thm2}. It only remains to treat the cases of equality in Proposition~\ref{thm1}.

\begin{proof}[Determination of equality cases in Proposition~\ref{thm1} ]
The idea is that, if equality holds  in Proposition~\ref{thm1}, and if   $\nu$ and $\mu_V$ have same barycenter, a situation that can be imposed by translating $\nu$ (provided we know that translation preserves equality cases), then we can apply Theorem~\ref{thm2} and conclude that $W_1(\mu_V, \nu)=0$, which implies $\nu = \mu_V$. Oddly enough, the converse requires also some work ; even the fact that there is equality when $\nu=\mu_V$ is not straightforward, and actually requires the convexity of $V$. 

We will prove the stronger result of Remark~\ref{rem:equality}.  Given a vector $v\in \R^n$, let us denote by $T_v \nu$ the translation by $v$ of the probability $\nu$; if $d\nu(x) = F(x) \, dx$ , then $dT_v \nu(x) = F(x-v)\, dx$.  The following Lemma is essential as it establishes the translation invariance of the inequality under study. 

\begin{lemma}[Translation invariance]\label{translation}
With the notation of Proposition~\ref{thm1}, we have, for any probability $\nu$ and any vector $v\in \R^n$, that
$$ H(T_v \nu ||\mu_V) - \mathcal W_{\costV}(\mu_V, T_v \nu) =   H(\nu ||\mu_V) - \mathcal W_{\costV}(\mu_V, \nu)  $$
\end{lemma}
\begin{proof}
To simplify the notation, we can assume that $\int e^{-V} = 1$, and also  that $d\nu(x) = f(x) d\mu_V(x) = f(x) e^{-V(x)} dx$. 
To treat to transportation term, we will need the following observation:
\begin{fact}
Given $v\in \R^n$, introduce $\tilde v = (0,v) \in \R^{2n}$. Let $\mu$ and $\nu$ be two probability measures on $\R^n$ and $\costV$ be the cost from Proposition~\ref{thm1}. If $\pi$ is a $\costV$-optimal coupling for $(\mu, \nu)$ then $T_{\tilde v} \pi$ is a $\costV$-optimal coupling for $(\mu, T_v \nu)$.  
\end{fact}
Let us prove this fact. The coupling condition is clear, so we only need to check that $T_{\tilde v} \pi$ is  $\costV$-optimal when $\pi$ is. Equivalently, by the characterization of optimality in terms of cyclical monotony (see~\cite[Chapter~5]{V}), it suffices to check that the support of $T_{\tilde v} \pi$ is $\costV$-cyclically monotone when the support of $\pi$ is 
$\costV$-cyclically monotone. Let $(x_1, y_1), \ldots , (x_k,y_k)$ be arbitrary points of $\R^{2n}$, with the convention that $(x_{k+1}, y_{k+1}):=(x_1,y_1)$. We have
\begin{eqnarray*}
\sum_{i=1}^k \costV(x_i, y_{i+1}+v) - \sum_{i=1}^{k} \costV(x_{i} , y_i+v)  & = & -\sum_{i=1}^k \nabla V(x_i) \cdot (y_{i+1}-y_i) \\
 &= &\sum_{i=1}^k \costV(x_i, y_{i+1}) - \sum_{i=1}^{k} \costV(x_{i} , y_i)  ,
\end{eqnarray*}
which shows that, indeed, the support of $T_{\tilde v} \pi$ is $\costV$-cyclically monotone if and only if the support of $\pi$ is $\costV$-cyclically monotone. 

With this Fact in hand, let us finish the proof of Lemma~\ref{translation}. Let $\pi$ be a $\costV$-optimal coupling for $(\mu_V, \nu)$. Then by the previous Fact we have that
$$\mathcal W_{\costV}(\mu_V, T_v \nu) -   \mathcal W_{\costV}(\mu_V, \nu) = \iint \big[\costV(x,y+v) -\costV(x,y)]\, d\pi(x) = \int \big[V(y+v) -V(y)]\, d\nu(y),$$
where we used that $\iint \nabla V(x)\cdot v \, d\pi(x,y) = \int \nabla V(x)\cdot v\, e^{-V(x)}\, dx = 0$.  

The entropic terms are easier to analyse. Since $dT_v\nu(x) = f(x-v) e^{-V(x-v)}\, dx = f(x-v) e^{-V(x-v)+V(x)} d\mu_V(x)$,  we have
\begin{multline*}
H(T_v\nu || \mu_V) - H(\nu ||\mu_V) = \int \big[\log(f(x-v)) - V(x-v) + V(x)]\, f(x-v) e^{-V(x-v)} \, dx \\
- \int  \log f(x) \log(f(x)) e^{-V(x)}\, dx =  \int \big[-V(x) +V(x+v)]\, d\nu(x) .
\end{multline*}
By subtracting the previous two equations, we obtain the conclusion of Lemma~\ref{translation}.
\end{proof}

Next, the r\^ole of the convexity of $V$ can be summarized as follows. 

\begin{lemma}\label{convexity_of_V}
Let $V:\R^n \to \R$ be locally Lipchitz function satisfying~\eqref{hyp} and  $\costV$  be the cost given by~\eqref{defcost2}, which is well defined  for  almost every $x\in \R^n$. If there exists an absolutely continuous probability measures $\mu$  with support $\R^n$ such that $\mathcal W_{\costV} (\mu, \mu)= 0$, then $V$ is convex on $\R^n$. Conversely, if $V$ is convex, then $\mathcal W_{\costV}(\mu, \mu)=0$ for every absolutely continuous probability measure $\mu$. 
\end{lemma}

\begin{proof}
Since $\cost(x,x)=0$, if $\mathcal W_{\cost}(\mu, \mu)=0$ then it means that the image $\pi$ of $\mu$ by the map $x \to (x,x)$ is an optimal coupling, and therefore its support is $\cost_V$-cyclically monotone. By the assumption on $\mu$, this implies that for (almost) all $x,y \in \R^n$ we have $\cost_V(x,y) + \cost_V(y,x) \ge \cost_V(x,x) + \cost_V(y,y) $, which rewrites as
\begin{equation}\label{monotony}
\big(\nabla V(y) - \nabla V(x) \big)\cdot (y-x) \ge 0.
\end{equation}
For a locally Lipschitz function (therefore also $W^{1,1}_{loc}$), this property implies that $V$ is convex. Indeed, we can consider $V_\epsilon = V \ast \eta_\epsilon$ where $\eta_\epsilon(x) = \epsilon^{-n} \eta(x/\epsilon)$ is an approximation of the identity in $\R^n$, with $\eta$ compactly supported. Then the property~\eqref{monotony} passes to $V_\epsilon$, which is now smooth, so that this property holds at every $(x,y)$ and this implies that $V_\epsilon$ is convex on $\R^n$, because it implies that the restriction of $V_\epsilon$ to any affine line has a nondecreasing derivative. We conclude by using that $V_\epsilon$ converges to $V$, point-wise as $\epsilon\to 0$. 

Conversely, the fact that $\cost_V(x,x)=0$ implies that $ \mathcal W_{\cost_V}(\mu, \mu)\le 0$ for any  absolutely continuous probability measure $\mu$. Since $\cost_V\ge 0$ when $V$ is convex, we get in this case that $ \mathcal W_{\cost_V}(\mu, \mu)= 0$. (One can also verify that when $V$ is convex, the set $\{(x,x)\; ; \ x\in \R^n\}$ is $\cost_V$-cyclically monotone.)

\end{proof}

We now have all the ingredients for the study of equality cases. If there is equality in~\eqref{ineq:thm1} for some $\nu$, then by Lemma~\ref{translation} there must be equality for any translated measure $T_v \nu$, $v\in \R^n$. But for  $v:= -\int x \, d\nu + \int x \, d\mu_V$ we have the centering condition $\int x\, dT_v \nu = \int x \, d\mu_V(x)$, and so we must have that $W_1(\mu_V, T_v \nu)=0$, that is $T_v \nu = \mu_V$ or equivalently $\nu = T_{-v} \mu_V$. This shows that for equality to hold,  $\nu$ must be a translate of $\mu_V$. 
But this in turn implies, again by Lemma~\ref{translation}, that there is also equality for $\nu=\mu_V$. Since $H(\mu_V \| \mu_V)=0$, we must have $\mathcal W_{\cost_V}(\mu_V, \mu_V)=0$. By Lemma~\ref{convexity_of_V} this implies that $V$ is convex.  

Conversely, if $V$ is convex, there is equality for $\nu=\mu$, because Lemma~\ref{convexity_of_V} ensures that $\mathcal W_{\cost_V}(\mu_V, \mu_V)=0$, and by Lemma~\ref{translation} we then have also equality for any translate of $\mu_V$.

\end{proof}

\section{Variance Brascamp-Lieb inequalities}

It is well known that linearization of transportation type inequalities give Poincar\'e type inequalities. One often uses the dual infimal convolution inequality~\eqref{ic} to perform the linearization, but one can do it also directly from the transportation inequality.  The procedure for linearizing the Wasserstein distance is standard, especially in the framework of the so-called ``Otto calculus" (see for instance~\cite{OV}). It is also known that  only the local behavior of the cost matters for linearizing a transport inequality (see for instance~\cite[Section 8.3]{GL}. However we did not find a reference for the precise situation studied here, and so we include for completeness the following statement.
% and its proof. 

\begin{lemma}\label{approx}
Let $\cost:\R^n \times \R^n \to \R^+$ be a function such that $\cost(y,y)=0$ and $\cost(x,y) \ge \delta_0 |x-y|^2$ for every $x,y\in \R^n$, for some $\delta_0>0$. Assume furthermore that for every $y$ there exists a nonnegative symmetric operator $H_y$ for which
$$\cost(y+h, y) =\frac12 H_y h \cdot h + |h|^2 o(1)$$ 
uniformly in $y$ on compact sets when $h\to 0$.  

Then, if $\mu$ is a probability measure on $\R^n$ and $g$ is a $C^1$ compactly supported function with $\int g \, d\mu =0$, we have
$$\liminf_{\eps \to 0} \frac1{\eps^2} \mathcal W_{\cost} (\mu, (1+\eps g)\, d\mu) \ge \frac12 \frac{\big(\int g  \, f\, d\mu \big)^2}{\int H^{-1}\nabla f\cdot \nabla f\, d\mu}$$
for any $C^1$ compactly supported function $f$. 
\end{lemma}

\begin{proof}
Given a (bounded) function $F$ on $\R^n$,  we introduce its  infimal convolution~\eqref{defic}  associated to our cost $\cost$ which satisfies: for every $(x,y) \in \R^n \times \R^n$, $Q_{\cost} (F)(y) - F(x)  \le \cost(x,y) $. It then follows from the definition of $\mathcal W_{\cost}$ that
$$\mathcal W_{\cost}(\mu, \nu) \ge \int Q_{\cost} (F)  \, d\nu - \int F\, d\mu.$$
In our situation where $\nu=(1+\eps g)\, d\mu$, ($\eps$ small enough and later tending to $0$) we pick $F=\eps f$ with $f$  of class $C^1$  compactly supported and $\int f\, d\mu =0$. Let us write
$$Q_c(\eps f)(y) =  \inf_{x} \big\{ \eps f(x) + \cost(x,y)\big\} = \inf_{h} \big\{ \eps f(y+h) + \cost(y+h,y)\big\}.$$
For any given $y$, let $h_\eps = h_{\eps, y}$ be a point where this infimum is achieved. Since the function $f$ is Lipschitz, of constant $M>0$ say, we have by our assumption on the cost that
$$\eps f(y)-\eps M |h_\eps| +\delta_0 |h_\eps|^2 \le\eps f(y+h_\eps)+c(y+h_\eps,y)\le \eps f(y) $$
so that 
$$|h_\eps |\le \frac{M}{\delta_0} \eps . $$
In other words,  $h_\eps$ tends to zero like $\eps$ uniformly in $y$. Also, since $f$ is continuous compactly supported, we can find (because the cost is nonnegative and large when points are far-apart) a bounded open set $\Omega$, which contains the support of $f$, such that $Q_{\cost} (\eps f)(y) \ge 0$ for every $y \in \R^n \setminus\Omega$. Consequently, we have
$$\mathcal W_{\cost}(\mu, (1+\eps g)\, d\mu) \ge  \int Q_{\cost} (\eps f)   \, (1+\eps g)\, d\mu \ge  \int_{\Omega} Q_{\cost} (\eps f)   \, (1+\eps g)\, d\mu.$$
We have, uniformly for $y$ in the bounded set $\Omega$,
$$ Q_{\cost} (\eps f) (y) = \eps f(y+h_\eps)+c(y+h_\eps,y) =  \eps f (y) + \eps \nabla f (y) \cdot h_\eps + \frac12 H_{y} h_\eps \cdot h_\eps  + o(\eps^2)$$
and so
$$ Q_{\cost} (\eps f) (y) \ge  \eps f (y) - \eps^2 \frac12 H_{y}^{-1} \nabla f (y) \cdot \nabla f (y)  + o(\eps^2).$$
After multiplying by $(1+\eps g)$, we can integrate on $\Omega$ using that the $o(\eps^2)$ is uniform in $y$:
$$\frac1{\eps^2}\mathcal W_{\cost}(\mu, (1+\eps g)\, d\mu) \ge  \int_\Omega f\, g \, d\mu -  \frac12 \int_\Omega  H_{y}^{-1} \nabla f \cdot \nabla f \, d\mu \; + \; o(1) $$
This implies, using that $\Omega$ contains the support of $f$, that 
$$\liminf_{\eps \to 0} \frac1{\eps^2} \mathcal W_{\cost} (\mu, (1+\eps g)\, d\mu) \ge   \int f\, g \, d\mu -  \frac12 \int  H_{y}^{-1} \nabla f \cdot \nabla f \, d\mu .$$
The result follows by homogeneity (replacing $f$ by $\lambda f$ and optimizing). 
\end{proof}

The linearization given in Lemma~\ref{approx} shows that the Brascamp-Lieb inequality follows immediately from Proposition~\ref{thm1} for $\nu =(1+\eps g) d\mu_V$ with $\int gd\mu_V=0$, when $\eps\to 0$. Indeed, without loss of generality we can assume that $D^2 V\ge 2\delta_0$ (by adding a small $\delta_0 |x|^2$ and later making $\delta_0\to 0$) so that the cost verifies also $\costV(x,y) \ge \delta_0 |x-y|^2$. Moreover, if $V$ is $C^2$ we see from the definition~\eqref{defcost2} of the cost that, when $h\to 0$, 
$$ \costV(y+h,y) - \frac12 D^2 V(y) h\cdot h = \int_0^1 \Big[ D^2 V(y+(1-t)h)  - D^2V(y)\Big]h \cdot h \, (1-t) \, dt =  |h|^2 o(1)$$
where the $o(1)$ uniform in $y$ on compact sets, since $D^2 V$ is uniformly continuous on compact sets. On the other hand, if $g$ is a $C^1$ compactly supported function with $\int g\, d\mu_V=0$, we have, 
$$H((1+\eps g) d\mu_V |\mu_V) = \frac12 \eps^2 \int g^2 \, d\mu_V + o(\eps^2).$$
So we find, by applying Proposition~\ref{thm1}  with $\nu= (1+\eps g) d\mu_V $ and the Lemma above with the choice $f=g$ (which is the optimal one in the present situation), at the limit, that
$$ \frac12 \frac{\big(\int g^2 \, d\mu_V \big)^2}{\int (D^2V(x))^{-1}\nabla g\cdot \nabla g\, d\mu_V} \le  \frac12 \int g^2 \, d\mu_V$$
which is the Brascamp-Lieb inequality~\eqref{hbl}. 

Let us apply the same procedure with inequality~\eqref{ineq:thm2} in Theorem~\ref{thm2}, the crucial point being that $ \costmin(h(\mu_V)\,|(y+h)-y| )$ behaves like $h(\mu_V)^2 |h|^2$ for $h$ small. So the cost satisfies, for $h\to 0$, 
\begin{multline*}
\ncost_V (y+h,y) = \cost_V (y+h,y) + c\, \costmin(h(\mu_V)\,|h| )=
\cost_V (y+h,y) + c\, h(\mu_V)^2 |h|^2 + o(h^2) \\ =  \frac12 \Big[ D^2 V(y) h\cdot h + 2c\, h(\mu_V)^2 \id \Big]h \cdot h + o(h^2) 
\end{multline*}
 where $c$ is a numerical constant. The same argument as before for $\nu = (1+\eps g )d\mu_V $ shows that if $g$ is a $C^1$ compactly supported function with $\int g\, d\mu_V = 0$ and
$$\int x g(x) \,  d\mu_V(x) = 0$$
we have
$$\int g^2 \, d\mu_V \le  \int \big( D^2V+c\, h(\mu_V)^2 \id)^{-1}\nabla g \cdot \nabla g \, d\mu_V$$
as claimed in Proposition~\ref{propRBL}. 

Finally, let us derive~Proposition~\ref{propQBL}. With the notation of the Proposition, for given $g$, denote $g_0:= g -\nabla V\cdot v_0 - c_0$. It is readily checked  by elementary calculus that for every vector $v_0$ and constant $c_0$ (so not only for the ones we have picked), if $g= g_0 +\nabla V\cdot v_0 + c_0$,
\begin{multline*}
A:= \int (D^2V(x))^{-1} \nabla g(x) \cdot \nabla g(x) \, d\mu_V(x) - \Var_{\mu_V} (g) \\
= \int (D^2V(x))^{-1} \nabla g_0(x) \cdot \nabla g_0(x) \, d\mu_V(x) - \Var_{\mu_V} (g_0) .
\end{multline*}
Next, for our choice of $v_0$ and $c_0$ observe that $\int g_0 \, d\mu_V=0$ and
$$\int x\, g_0(x) \, d\mu(x) =0,$$
since, in the standard basis, writing $x_j=x\cdot e_j$ for $j=1, \ldots, n$, we have
$$\int  x_j \nabla V(x)\cdot v_0\,  d\mu_V(x) = \int e_j \cdot v_0 \, d\mu_V = e_j \cdot v_0  = \int x_j (g(x)-c_0) \, d\mu_V(x).  $$
So by Proposition~\ref{propRBL} we find
\begin{eqnarray*}
A &\ge  &\int (D^2V)^{-1} \nabla g_0 \cdot \nabla g_0 \, d\mu_V -  \int \big( D^2V+ c \lambda(\mu_V)\id)^{-1}\nabla g_0 \cdot \nabla g_0 \, d\mu_V\\
&=&  c\lambda(\mu_V)\int (D^2V)^{-1} \big( D^2V+ c \lambda(\mu_V)\id)^{-1}\nabla g_0 \cdot \nabla g_0 \, d\mu_V\\
\end{eqnarray*}
From this bound, we can proceed in two different ways. First, we can use a uniform lower bound and combine it with the Brascamp-Lieb inequality,
$$A \ge \frac{c \lambda(\mu_V)}{\sup_x  \lambda_{\rm max}(x) + c \lambda(\mu_V)}  \int (D^2 V)^{-1}\nabla g_0 \cdot \nabla g_0 \, d\mu_V
\ge \frac{c \lambda(\mu_V)}{\sup _x   \lambda_{\rm max}(x) + c \lambda(\mu_V)}  \int |g_0|^2 \, d\mu_V $$ 
Otherwise, using again that $D^2 V \le \lambda_{\rm max} \id$, we can use H\"older's inequality, to arrive to
$$ 
A \ge  \frac{\big(\int |\nabla g_0| \, d\mu_V\big)^2}{\int  \lambda_{\rm max} (  \lambda_{\rm max} + c \lambda(\mu_V)) \, d\mu_V} .
 $$
But~\eqref{equivcheegerpoincare} implies that  $\int |\nabla g_0| \, d\mu_V \ge c \sqrt{\lambda(\mu_V)} \int |g_0| \, d\mu_V$ 
for some numerical constant $c>0$. This ends the proof of Proposition~\ref{propQBL}.

\vskip1cm 
\noindent Dario Cordero-Erausquin \\
  Institut de Math\'ematiques de Jussieu, \\ Universit\'e Pierre et Marie Curie - Paris 6, \\
  75252 Paris Cedex 05, France \\
 \verb?dario.cordero@imj-prg.fr?

\end{document}